\documentclass[12pt]{article}

\usepackage{graphicx,hyperref}

\usepackage{graphicx}
\usepackage{multirow}   
\usepackage{bm}         
\usepackage{amsmath}    
\usepackage{amsfonts}   
\usepackage{amssymb}    
\usepackage{calligra}   
\usepackage{mathrsfs}   
\usepackage{amsthm}     
\usepackage{natbib}
\usepackage{url}        
\usepackage{graphicx}
\usepackage{multirow}
\usepackage{subfigure}
\usepackage{enumerate}
\usepackage{rotating}

\usepackage{mathtools}

\numberwithin{equation}{section}
\theoremstyle{plain}
\newtheorem{theorem}{Theorem}[section]
\newtheorem{lemma}[theorem]{Lemma}

\newcommand{\be}{\begin{equation}}
\newcommand{\ee}{\end{equation}}
\newcommand{\beqa}{\begin{eqnarray*}}
\newcommand{\eeqa}{\end{eqnarray*}}
\newcommand{\beqn}{\begin{eqnarray}}
\newcommand{\eeqn}{\end{eqnarray}}
\newcommand{\ba}{\begin{array}}
\newcommand{\ea}{\end{array}}
\newcommand{\bc}{\begin{center}}
\newcommand{\ec}{\end{center}}
\newcommand{\btab}{\begin{tabular}}
\newcommand{\etab}{\end{tabular}}

\newcommand{\mb}{\makebox}
\newcommand{\lt}{\left}
\newcommand{\rt}{\right}

\newcommand{\ld}{\ldots}
\newcommand{\cd}{\cdot}
\newcommand{\st}{\stackrel}

\newcommand{\iy}{\infty}

\newcommand{\mc}[1]{\mathcal{#1}}

\newcommand{\Ind}{1\!\mathrm{l}}

\newcommand{\ind}{\, \raise-2pt\hbox{$\st{\mb{\scriptsize ind}}{\sim}$}\, }
\newcommand{\iid}{\, \raise-2pt\hbox{$\st{\mb{\scriptsize iid}}{\sim}$}\,}

\renewcommand{\P}{\mathrm{P}}

\newcommand{\E}{\mathrm{E}}

\newcommand{\tr}{\mathrm{tr}}

\newcommand{\bX}{\bm{X}}
\newcommand{\bx}{\bm{x}}

\newcommand{\bGamma}{\bm{\Gamma}}
\newcommand{\bDelta}{\bm{\Delta}}
\newcommand{\bSigma}{\bm{\Sigma}}
\newcommand{\bOmega}{\bm{\Omega}}

\newcommand{\bmA}{\bm{A}}

\newcommand{\bmI}{\bm{I}}

\newcommand{\RR}{{\mathbb R}}

\mathchardef\given="626A
\mathcode`:="603A

\long\def\beginskip#1\endskip{}
\def\endskip{}

\newcommand{\eig}{\mathrm{eig}}

\title{Bayesian estimation of a sparse precision matrix}

\author{Sayantan Banerjee\footnote{Corresponding author at: Department of Statistics, North Carolina State University, 5219 SAS Hall, 2311 Stinson Drive, Raleigh, NC 27695-8203. Tel: +1-919-699-8773. e-mail: sayantan.banerjee.isi@gmail.com}, Subhashis Ghosal\\ {\it North Carolina State University}}
\date{}

\begin{document}
\maketitle


\begin{abstract}
  We consider the problem of estimating a sparse precision matrix of a multivariate Gaussian distribution, including the case where the dimension $p$ is large. Gaussian graphical models provide an important tool in describing conditional independence through presence or absence of the edges in the underlying graph. A popular non-Bayesian method of estimating a graphical structure is given by the graphical lasso. In this paper, we consider a Bayesian approach to the problem. We use priors which put a mixture of a point mass at zero and certain absolutely continuous distribution on off-diagonal elements of the precision matrix. Hence the resulting posterior distribution can be used for graphical structure learning. The posterior convergence rate of the precision matrix is obtained. The posterior distribution on the model space is extremely cumbersome to compute. We propose a fast computational method for approximating the posterior probabilities of various graphs using the Laplace approximation approach by expanding the posterior density around the posterior mode, which is the graphical lasso by our choice of the prior distribution. We also provide estimates of the accuracy in the approximation.
\end{abstract}


\section{Introduction}
Statistical inference on large covariance or precision matrix (inverse of covariance matrix) is a topic of growing interest in recent times. Often the dimension $p$ grows with the sample size $n$ and even $p$ can be bigger than $n$. Data of this type are frequently encountered in fMRI, spectroscopy, gene array expressions and so on. Estimation of the covariance or precision matrix is of special interest because of their importance in methods like principal component analysis (PCA), linear discriminant analysis (LDA), etc. In cases where $p > n$, the sample covariance matrix is necessarily singular, and hence an estimator of the precision matrix cannot be obtained by inverting it. Therefore we need to resort to other techniques for handling the high-dimensional problems.

\par Regularization methods for estimation of the sample covariance or precision matrix have been proposed and studied in recent literature for high-dimensional problems. These include banding, thresholding, tapering and penalization based methods; for example, see \cite{ledoit2004well, huang2006covariance, yuan2007model, bickel2008covariance, bickel2008regularized, karoui2008operator, friedman2008sparse, rothman2008sparse, lam2009sparsistency, rothman2009generalized, cai2010optimal, cai2011constrained}. The primary goal of these regularization based methods is to impose a sparsity structure in the matrix. Most of these methods are applicable to situations where there is a natural ordering in the underlying variables, for example in data from time series, spatial data, etc., so that variables which are far off from each other have smaller correlations or partial correlations. In high-dimensional situations for data arising from genetics or econometrics, a natural ordering of the underlying variables may not always be readily available and hence estimation methods which are invariant to the ordering of the variables are desirable.

\par For estimation of a sparse inverse covariance matrix, graphical models \citep{lauritzen1996graphical} provide an excellent tool, as the conditional dependence between the component variables is captured an undirected graph; see \cite{dobra2004sparse,meinshausen2006high,yuan2007model,friedman2008sparse}. There are several methods in the frequentist literature for the estimation of the precision matrix through graphical models. These methods include minimization of the penalized log-likelihood of the data with a lasso type penalty on the elements of the precision matrix. Several algorithms have been developed in the literature to solve the above optimization problem, including coordinate descent based algorithm for the lasso, which is popularly known as the graphical lasso \citep{meinshausen2006high, friedman2008sparse, banerjee2008model, yuan2007model, guo2011joint, witten2011new}. Other methods include the Sparse Permutation Invariant Covariance Estimator (SPICE) \citep{rothman2008sparse}.

\par Frequentist behavior of Bayesian methods in the context of high dimensional covariance matrix estimation have been studied only by a few authors. \cite{ghosal2000asymptotic} studied asymptotic normality of posterior distributions for exponential families, which include the normal model with unknown covariance matrix, when the dimension $p \rightarrow \infty$, but restricting to $p \ll n$. Recently, \cite{pati2012factor} considered sparse Bayesian factor models for dimensionality reduction in high dimensional problems and showed consistency in the $L_2$-operator norm (also known as the spectral norm) by using a point mass mixture prior on the factor loadings, assuming such a factor model representation of the true covariance matrix.

\par Bayesian methods for inference using graphical models have also been developed, as in \cite{roverato2000cholesky, atay2005monte, letac2007wishart}. A conjugate family of priors, known as the $G$-Wishart prior \citep{roverato2000cholesky} have been developed for incomplete decomposable graphs. The equivalent prior on the covariance matrix is termed as the hyper inverse Wishart distribution in \cite{dawid1993hyper}. \cite{letac2007wishart} introduced a more general family of conjugate priors for the precision matrix, known as the $W_{P_G}$-Wishart family of distributions, which also has the conjugacy property. The properties of this family of distributions, including expressions for the Bayes estimators were further explored in \cite{rajaratnam2008flexible}. Recently \cite{banerjeeghoshal2013} studied posterior convergence rates for a $G$-Wishart prior inducing a banding structure, where the true precision matrix need not have the banding structure.

\par \cite{wang2012bayesian} developed a Bayesian version of the graphical lasso, putting Laplace priors on the off-diagonal elements of the precision matrix and exponential priors on the diagonals. Similar in lines with the Bayesian lasso \citep{park2008bayesian}, the posterior mode in this case coincides with the graphical lasso estimate. A block Gibbs sampler is also developed for sampling from the resulting posterior. However, the Bayesian graphical lasso does not introduce any sparsity in the graphical structure because of the absence of a point mass at zero in the prior distribution for the off-diagonal elements. On the other hand, if point masses are introduced, the resulting posterior distribution on the structure of the graph becomes extremely difficult to compute based on the traditional reversible jump Markov chain Monte Carlo method.

\par In this paper, we derive posterior convergence rates for the Bayesian graphical lasso prior in terms of the Frobenius norm under appropriate sparsity conditions. For computing the posterior distribution, we propose a Laplace approximation based method to compute the posterior probability of different graphical structures. Such Laplace approximations based methods have been developed for variable selection in regression models; for example, see \cite{yuan2005efficient, mckay2013fast}. The lasso type penalty on the elements lead to non-differentiability of the integrand, when the graphical lasso sets an off-diagonal entry to zero, but the model includes that off-diagonal entry as a free variable. We shall call such models non-regular following the terminology used by \cite{yuan2005efficient} for variable selection in linear regression models. We show that the posterior probability of non-regular models are substantially smaller than their regular counterparts and hence in comparison may be ignored from consideration. We also estimate the error in the Laplace approximation for regular models. 

\par The paper is organized as follows. In the next section, we introduce notations and discuss preliminaries on graphical models required for the other sections of the paper. In Section 3, we state model assumptions and specify the prior distribution on the underlying parameters, derive the form of the posterior and obtain the posterior convergence rate using the general theory developed in \cite{ghosal2000convergence}. In Section 4, we develop the approximation of the posterior probabilities for different graphical models and discuss the issue of non-regular graphical models. We also show that the error in approximation of the posterior probabilities using the Laplace approximation is asymptotically negligible under appropriate conditions. A simulation study is performed in the Section 5 followed by a real data example in Section 6. Proofs of main results and additional lemmas are included in the Appendix.

\section{Notations and preliminaries}

An undirected graph $G$ comprises of a non-empty set of $p$ vertices $V$ indexing the components of a $p$-dimensional random vector along with an edge-set $E$ defined by $E \subset \{(i,j)\in V \times V:\, i<j\}$. Let $\bX = (X_1,\ldots, X_p)^T$ be $\mathrm{N}_p(\bm{0},\bOmega^{-1})$, where the precision matrix $\bOmega = (\!(\omega_{ij})\!)$ is such that $(i,j) \not \in E$ implies $\omega_{ij} = 0$. We then say that $\bX$ follows a Gaussian graphical model (GGM) with respect to the graph $G$. Since the absence of an edge between $i$ and $j$ implies conditional independence of $X_i$ and $X_j$ given $(X_r : r \neq i,j)$, a GGM serve as an excellent tool in representing the sparsity structure in the precision matrix. Following the notation in \cite{letac2007wishart}, the canonical parameter $\bOmega$ is restricted to $\mathcal{P}_G$, where $\mathcal{P}_G$ is the cone of positive definite symmetric matrices of order $p$ having zero entry corresponding to each missing edge in $E$. We also denote the linear space of symmetric matrices of order $p$ by $\mathcal{M}$, and $\mathcal{M}^+ \subset \mc{M}$ to be the cone of positive definite matrices of order $p$. Corresponding to each GGM $G = (V,E)$, we define the set $\mathcal{V} =  \{(i,j)\in V \times V:\, i= j,  \mathrm{or}\, (i,j) \in E\}$.
\par By $t_n = O(\delta_n)$ (respectively, $o(\delta_n)$), we  mean that $t_n/\delta_n$ is bounded (respectively, $t_n/\delta_n \rightarrow 0$ as $n\to\iy$). For a random sequence $X_n$, $X_n = O_P(\delta_n)$ (respectively, $X_n = o_P(\delta_n)$) means that $\P(|X_n| \leq M\delta_n) \rightarrow 1$ for some constant $M$ (respectively, $\P(|X_n| < \epsilon\delta_n) \rightarrow 1$ for all $\epsilon > 0$). For numerical sequences $r_n$ and $s_n$, by $ r_n \ll s_n$ (or, $r_n \gg s_n)$ we mean that $r_n = o(s_n)$, while by $s_n \gtrsim r_n$ we mean that $r_n = O(s_n)$. By $r_n\sim s_n$ we mean that $r_n/s_n\to 1$. The indicator function is denoted by $\Ind$.
Vectors are represented in bold lowercase English or Greek letters with the components of a vector by the corresponding non-bold letters, that is, for $\bx \in \RR^p$, $\bx = (x_1,\ldots,x_p)^T$. For a vector $\bx \in \RR^p$, we define the following vector norms:
$\|\bx\|_r = \left(\sum_{j=1}^p|x_{j}|^r\right)^{1/r}$, $\|\bx\|_\infty = \mathop{\max }_{j}|x_{j}|$. 
Matrices are denoted in bold uppercase English or Greek letters, like $\bmA = (\!(a_{ij})\!)$, where $a_{ij}$ stands for the $(i,j)$th entry of $\bm{A}$. The identity matrix of order $p$ will be denoted by $\bm{I}_p$.
If $\bm{A}$ is a symmetric $p\times p$ matrix, let $\eig_1(\bmA) \leq \ld \leq \eig_p(\bm{A})$ stand for its eigenvalues and let the trace of $\bm{A}$ be denoted by $\tr(\bm{A})$.
Viewing $\bm{A}$ as a vector in $\RR^{p^2}$, we define $L_r,\, 1 \leq r < \infty$ and $L_{\infty}$-norms on $p\times p$ matrices as
\begin{equation*}
\|\bmA\|_r =\lt(\sum_{i=1}^p\sum_{j=1}^p |a_{ij}|^r\rt)^{1/r}, \; 1\le r<\iy, \quad \|\bmA\|_\infty = \mathop{\max }_{i,j}|a_{ij}|.
\end{equation*}
Note that $\|\bmA\|_2 = \sqrt{\tr(\bmA^T\bmA)}$, the Frobenius norm. Viewing $\bm{A}$ an operator from $(\RR^p,\|\cd\|_r)$ to $(\RR^p, \|\cd\|_s)$, where $1\le r,s\le\iy$, we can also define, $\|\bmA\|_{(r,s)} = \mbox{sup}(\|\bmA\bx\|_s:\|\bx\|_r = 1).$ We refer to the norm $\|\cd\|_{(r,r)}$ as the $L_r$-operator norm.
This gives the $L_2$-operator norm of $\bmA$ as
\begin{equation*}
\|\bmA\|_{(2,2)} = [\max\{\eig_i(\bmA^T\bmA):1\le i\le p\}]^{1/2}.
\end{equation*}
For symmetric matrices, $\|\bmA\|_{(2,2)}= \max\{|\eig_i(\bmA)|:1\le i\le p\}$.  For symmetric matrices $\bmA$ and $\bm{B}$ of order $p$, we have the following:
\begin{equation}
\label{eqn:matrixnorm}
\begin{split}
&\|\bmA\|_\infty \leq \|\bmA\|_{(2,2)} \leq \|\bmA\|_2 \leq p\|\bmA\|_{\infty}, \\
&\|\bmA\bm{B}\|_2 \leq \|\bmA\|_{(2,2)}\|\bm{B}\|_2,\, \|\bmA\bm{B}\|_2 \leq \|\bmA\|_2\|\bm{B}\|_{(2,2)}.
\end{split}
\end{equation}
$\bm{A}^{1/2}$ stands for the unique positive definite square root of a positive definite matrix $\bm{A}$. For two matrices $\bm{A}$ and $\bm{B}$, we say that $\bm{A}\ge \bm{B}$ (respectively, $\bm{A}> \bm{B}$) if $\bm{A}-\bm{B}$ is nonnegative definite (respectively, positive definite). Thus $\bmA > \bm{0}$ for a positive definite matrix $\bmA$, where $ \bm{0}$ stands for the zero matrix. We denote sets in non-bold uppercase English letters. The cardinality of a set $T$, that is, the number of elements in $T$ is denoted by $\#T$. We define the symmetric matrix $\bm{E}_{(i,j)} = (\!(\Ind_{\{(i,j),(j,i)\}}(l,m))\!)$.


The Hellinger distance between two probability densities $q_1$ and $q_2$ is given by $h(q_1,q_2) = \|\sqrt{q_1}-\sqrt{q_2}\|_2$.


For a subset $A$ of a metric space $(S,d)$, $N(\epsilon,A,d)$ denote the $\epsilon$-covering number of $A$ with respect to $d$, that is, the minimum number of $d$-balls of size $\epsilon$ is $S$ needed to cover $A$.

\section{Model, prior and posterior concentration}
Consider $n$ independent random samples $\bX_1,\ldots,\bX_n$ from $\mathrm{N}_p(\bm{0},\bSigma)$, where $\bSigma$ is nonsingular and the precision matrix $\bOmega = \bSigma^{-1}$ is sparse. The problem is to estimate $\bOmega$ and to learn the underlying graphical structure. We denote the natural unbiased estimator of $\bSigma$ by $\widehat{\bSigma} = n^{-1}\sum_{i=1}^{n}\bX_i\bX_i^T$ .

\par The graphical lasso produces sparse solutions for the precision matrix, in similar lines to that of the lasso in case of linear regression. The graphical lasso estimator minimizes two times the penalized negative log-likelihood 
\begin{equation}
-\log \det(\bOmega) + \mathrm{tr}(\widehat{\bSigma}\bOmega) + \rho\|\bOmega\|_1,
\end{equation}
over the class of positive definite matrices, and $\rho \geq 0$ acts as the penalty parameter. \cite{rothman2008sparse} derived frequentist convergence rates of the penalized estimator under some sparsity assumptions on the true precision matrix. More specifically, consider the following class of positive definite matrices of order $p$:
\begin{equation}
\label{matrixclass}
\begin{split}
\mathcal{U}(\varepsilon_0,s) &= \left\{\bOmega: \#\{(i,j): \omega_{ij} \neq 0, 1 \leq i < j \leq p\} \leq s, \right. \\
&\qquad  \left. 0< \varepsilon_0 \leq \eig_1(\bOmega) \leq \eig_p(\bOmega) \leq \varepsilon_0^{-1} < \infty \right\} .
\end{split}
\end{equation}
Though \cite{rothman2008sparse} considered penalizing only the off-diagonal elements of $\bOmega$, some modification of the proof of their result leads to the same convergence rate for the graphical lasso estimator, obtained by additionally penalizing the diagonal elements. Let us denote $\bOmega^*$ as the graphical lasso estimator based on a sample of size $n$ from a $p$-dimensional Gaussian distribution with precision matrix $\bOmega_0 \in \mathcal{U}(\varepsilon_0,s)$, where $\mathcal{U}(\varepsilon_0,s)$ is given by (\ref{matrixclass}). Then, it follows from Theorem 1 in \cite{rothman2008sparse} that the rate of convergence of $\bOmega^*$ is $n^{-1/2}(p+s)^{1/2}\log p$. By the triangle inequality,
\begin{equation*}
\|\bOmega^*\|_{(2,2)} \leq \|\bOmega_0\|_{(2,2)} + \|\bOmega^* - \bOmega_0\|_{(2,2)}.
\end{equation*}
Also, the triangle inequality and sub-multiplicative property for matrix operator norms gives,
\begin{eqnarray}
\|\bOmega^{*-1}\|_{(2,2)} &\leq & \|\bOmega_0^{-1}\|_{(2,2)} + \|\bOmega^{*-1} - \bOmega_0^{-1}\|_{(2,2)} \nonumber \\
&\leq & \|\bOmega_0^{-1}\|_{(2,2)} + \|\bOmega_0^{-1}\|_{(2,2)}\|\bOmega^* - \bOmega_0\|_{(2,2)}\|\bOmega^{*-1}\|_{(2,2)}. \nonumber
\end{eqnarray}
Thus, we get,
\begin{equation*}
\|\bOmega^{*-1}\|_{(2,2)} \leq \frac{\|\bOmega_0^{-1}\|_{(2,2)}}{1 - \|\bOmega_0^{-1}\|_{(2,2)}\|\bOmega^* - \bOmega_0\|_{(2,2)}}.
\end{equation*}
Now, we have, $\|\bOmega_0\|_{(2,2)} \leq \varepsilon_0^{-1}$ by assumption, and it follows from Theorem 1 in \cite{rothman2008sparse} that $\|\bOmega^* - \bOmega_0\|_{2} = o_P(1)$ as $n \rightarrow \infty$. Noting that $\|\bOmega^* - \bOmega_0\|_{(2,2)} \leq \|\bOmega^* - \bOmega_0\|_{2}$, we get,
\begin{equation}
\label{eqn:glassoeigbound}
\|\bOmega^*\|_{(2,2)} = O_P(1),\, \|\bOmega^{*-1}\|_{(2,2)} = O_P(1).
\end{equation}

In the Bayesian context, \cite{wang2012bayesian} introduced the graphical lasso prior, which uses exponential distributions on diagonal elements and Laplace density $\lambda e^{-\lambda|x|}/2$ on off-diagonal elements, all independently of each other, and finally imposes a positive definiteness constraint. The graphical lasso prior has a drawback that it puts absolutely continuous priors on the elements of the precision matrix, and hence the posterior probabilities of the event $\{\omega_{ij} = 0\}$ is always exactly zero.

\par \cite{wang2012bayesian} also mentioned an extension of the graphical lasso by putting an additional level of prior on the underlying graphical model structure using point mass priors on the events corresponding to the absence of an edge in the edge-set $E$, although did not develop the method. We put point-mass prior on the events $\{\omega_{ij} = 0\}$ to make posterior inference about the sparse structure of the underlying graphical model. Define $\bGamma = (\gamma_{ij}: 1 \leq i < j \leq p)$ to be a $\binom{p}{2}$ vector of edge-inclusion indicator, that is,
\begin{equation}
\gamma_{ij} = \Ind\{(i,j) \in E\}, \, 1 \leq i < j \leq p.
\end{equation}
Similar to the Bayesian graphical lasso prior, given the underlying graphical structure, we put a Laplace prior on the non-zero off-diagonal elements of the precision matrix and for the diagonal elements we have a exponential prior, overall maintaining the positive definiteness of the parameter. Then the joint prior density on $\bOmega$ is given by,
\begin{equation}
\label{prior:omega}
p(\bOmega|\bGamma) \propto \prod_{\gamma_{ij} = 1}\left\{\exp(-\lambda|\omega_{ij}|)\right\}\prod_{i=1}^{p}\left\{\exp\left(-\lambda\omega_{ii}/2\right)\right\} \Ind_{\mathcal{M}^+}(\bOmega).
\end{equation}
We propose two different priors on the graphical structure indicator $\bGamma$. The edge indicators $\gamma_{ij}, 1 \leq i < j \leq p$ are considered to be independent and identically distributed (i.i.d) Bernoulli$(q)$ random variables, and conditioned to the restriction that the model size $\sum_{1 \leq i < j \leq p}\gamma_{ij}$ does not exceed $\bar{R}$. For some $a_1, a_2 >0$, the prior distribution on $\bar{R}$ is assumed to satisfy
\begin{equation}
\P(\bar{R} > a_1m) \leq e^{-a_2m\log m}.
\end{equation}
This prior is similar to that used by \cite{castillo2012needles}, which chooses the model size first according to a distribution with a similar tail decay and then subsets are selected randomly with equal probability. We can also specify the individual priors on $\gamma_{ij}$ the same as above, but now truncating the model size to some fixed $\bar{r}$, where $\bar{r}$ is chosen so as to satisfy the metric entropy condition required for posterior convergence.

\par Thus, in the first situation, the prior on the graphical structure indicator $\bGamma$, given $\bar{R}$, is given by,
\begin{equation}
\label{prior:graphhier}
p(\bGamma \mid \bar{R}) \propto  q^{\#\bGamma}(1-q)^{\binom{p}{2}-\#\bGamma}\Ind(\#\bGamma \leq \bar{R}),
\end{equation}
leading to
\begin{equation}
\label{prior:gamma1}
p(\bGamma) \propto q^{\#\bGamma}(1-q)^{\binom{p}{2}-\#\bGamma}\P(\bar{R} \geq \#\bGamma).
\end{equation}
In the second case, the prior on $\bGamma$ is simply given by
\begin{equation}
\label{prior:gamma2}
p(\bGamma) \propto q^{\#\bGamma}(1-q)^{\binom{p}{2}-\#\bGamma}\Ind(\#\bGamma \leq \bar{r}).
\end{equation}

Smaller values of $q$ prefer graphical models with fewer number of edges, hence inducing more sparsity in the precision matrix.

Due to the positive definiteness constraint on the parameter $\bOmega$, the normalizing constant corresponding posterior distribution of the graphical model becomes intractable and hence was not explored in \cite{wang2012bayesian}. One possible solution is to employ a reversible jump Markov chain Monte Carlo (RJMCMC) algorithm, which jumps from models of varying dimensions to evaluate the posterior probabilities. As there are as many as $2^{\binom{p}{2}}$ possible models, the posterior model probabilities estimated by RJMCMC visits are extremely unreliable. We consider a radically different approach to posterior computation based on Laplace approximations, elaborated in the next section.

Under the above prior specifications, the joint posterior distribution of $\bOmega$ and $\bGamma$ given the data $\bX^{(n)} = (\bX_1,\ldots,\bX_n)$ is given by
\begin{eqnarray}
\label{eqn:jointposterior}
p\{\bOmega,\bGamma|\bX^{(n)}\} &\propto & p\{\bX^{(n)}|\bOmega,\bGamma\}p(\bOmega|\bGamma)p(\bGamma) \nonumber \\
&=& (2\pi)^{np/2}\{\mathrm{det} (\bOmega)\}^{n/2}\exp\left\{-n\,\mathrm{tr}(\widehat{\bSigma}\bOmega)/2\right\} \nonumber \\
&& \times \prod_{\gamma_{ij}=1}\left\{\lambda\exp(-\lambda|\omega_{ij}|)/2\right\}\prod_{i=1}^{p}\left\{\lambda\exp\left(-\lambda\omega_{ii}/2\right)/2\right\} \nonumber \\
&& \times \, p(\bGamma) \Ind_{\mathcal{M}^+}(\bOmega).
\end{eqnarray}
Thus,
\begin{equation}
p\{\bOmega,\bGamma|\bX^{(n)}\} \propto C_{\bGamma}Q\{\bOmega,\bGamma|\bX^{(n)}\},
\end{equation}
where
\begin{eqnarray}
C_{\bGamma} &=& (2\pi)^{np/2}q^{\#\bGamma}(1-q)^{\binom{p}{2}-\#\bGamma}(\lambda/2)^{p+\#\bGamma}\beta(\bGamma), \nonumber \\
\beta(\bGamma) &=& \begin{cases} \P(\bar{R} \geq \#\bGamma),\, \mbox{for prior as in (\ref{prior:gamma1})}, \nonumber \\
\Ind(\#\bGamma \leq \bar{r}),\, \mbox{for prior as in (\ref{prior:gamma2})}, \end{cases} \nonumber \\
Q\{\bOmega,\bGamma|\bX^{(n)}\} &=& \{\mathrm{det} (\bOmega)\}^{n/2}\exp\left\{-n\,\mathrm{tr}(\widehat{\bSigma}\bOmega)/2\right\} \prod_{\gamma_{ij}=1}\left\{\exp(-\lambda|\omega_{ij}|)\right\}\nonumber \\
&& \times \prod_{i=1}^{p}\left\{\exp\left(-\lambda\omega_{ii}/2\right)\right\}\Ind_{\mathcal{M}^+}(\bOmega).
\end{eqnarray}

The following result gives posterior convergence rate as $n \rightarrow \infty$. We assume that the true model is sparse, as given by the class of positive definite matrices in (\ref{matrixclass}).

\begin{theorem}
\label{theorem:postconvrate}
Let $\bX^{(n)} = (\bX_1,\ldots,\bX_n)$ be a random sample from a $p$-dimensional Gaussian distribution with mean $\bm{0}$ and precision matrix $\bOmega_0 \in \mathcal{U}(\varepsilon_0,s)$ for some $0 < \varepsilon_0 < \infty $ and $0 \leq s \leq p(p-1)/2$. Also assume that the prior distributions $p(\bOmega \mid \bGamma)$ and $p(\bGamma)$ as in (\ref{prior:omega}) and (\ref{prior:gamma1}) or (\ref{prior:gamma2}) with $q < 1/2$. Then the posterior distribution of $\bOmega$ satisfies
\begin{equation}
\E_0\left[\P\left\{\|\bOmega - \bOmega_0\|_2 > M\epsilon_n \mid \bX^{(n)}\right\}\right] \rightarrow 0,
\end{equation}
for $\epsilon_n = n^{-1/2}(p+s)^{1/2}(\log p)^{1/2}$ and a sufficiently large constant $M > 0$.
\end{theorem}
The proof uses the general theory of posterior convergence of \cite{ghosal2000convergence} and will be given in the appendix. The above posterior convergence rate matches exactly with the frequentist convergence rate of the penalized estimator obtained in \cite{rothman2008sparse}. 

Note that, Theorem \ref{theorem:postconvrate} gives $\|\bOmega - \bOmega_0\|_2 = O(\epsilon_n)$ with posterior probability tending to one in probability and from \cite{rothman2008sparse} it follows that, $\|\bOmega^* - \bOmega_0\|_2 = O_P(\epsilon_n)$, where $\bOmega^*$ is the graphical lasso estimate. Hence, by the triangle inequality, $\|\bOmega - \bOmega^*\|_2 = O(\epsilon_n)$ with posterior probability tending to one in probability. This gives,
\begin{equation}
\label{eqn:postratio}
\frac{\int_{\|\bOmega-\bOmega^*\|_2 \leq \epsilon_n} \exp\{-n\,h(\bOmega)/2\} \prod_{(i,j) \in \mathcal{V}_{\bGamma}}d\omega_{ij}}{\int_{\bOmega \in \mathcal{M}^+} \exp\{-n\,h(\bOmega)/2\} \prod_{(i,j) \in \mathcal{V}_{\bGamma}}d\omega_{ij}} \rightarrow 1.
\end{equation}

\section{Posterior Computation}

The marginal posterior density of the graphical structure indicator $\bGamma$ can be obtained by integrating out elements of the precision matrix in the joint posterior density in (\ref{eqn:jointposterior}), to get
\begin{equation}
\label{eqn:postprob}
p\{\bGamma|\bX^{(n)}\} \propto C_{\bGamma}\int_{\bOmega \in \mathcal{M}^+} \exp\{-n\,h(\bOmega)/2\} \prod_{(i,j) \in \mathcal{V}_{\bGamma}}d\omega_{ij},
\end{equation}
where
\begin{equation}
\label{eqn:hfunc}
h(\bOmega) = -\log \mathrm{det}(\bOmega) + \mathrm{tr}(\widehat{\bSigma}\bOmega) + \frac{2\lambda}{n}\sum_{\gamma_{ij}=1}|\omega_{ij}| + \frac{\lambda}{n}\sum_{i=1}^{p}\omega_{ii}.
\end{equation}
Note that $h(\bOmega)$ is minimized at $\bOmega = \bOmega^{*}$, the graphical lasso estimate corresponding to the penalty parameter $\rho = \lambda/n$.
The marginal posterior of $\bGamma$ is, however, intractable. We give an approximate method for the posterior probability computations of various models using Laplace approximation. The Laplace approximation requires expanding the integrand in (\ref{eqn:postprob}) around the maximum, which in this case, coincides with the graphical lasso solution.
\subsection{Approximating model posterior probabilities}
Define $\bDelta = \bOmega - \bOmega^{*} = (\!(u_{ij})\!)$, where $\bOmega^{*}$ is the graphical lasso solution corresponding to the underlying graphical model structure and penalty parameter $\lambda/n$. Then, 
\begin{equation}
\begin{split}
\label{eqn:postprob2}
p\{\bGamma|\bX\} &\propto C_{\bGamma}\exp\{-n\,h(\bOmega^*)/2\} \left\{\mathrm{det}(\bOmega^*)\right\}^{-n/2} \\
&\qquad \times \int_{\bDelta + \bOmega^* \in \mathcal{M}^+} \exp\{-n\,g(\bDelta)/2\}\prod_{(i,j) \in \mathcal{V}_{\bGamma}}du_{ij},
\end{split}
\end{equation}
where $g(\bDelta)$ is
\begin{equation}
-\log \mathrm{det}(\bDelta + \bOmega^*) + \mathrm{tr}(\widehat{\bSigma}\bDelta) + \frac{2\lambda}{n}\sum_{\gamma_{ij}=1}(|u_{ij} + \omega^*_{ij}|- |\omega^*_{ij}|) + \frac{\lambda}{n}\sum_{i=1}^{p}u_{ii}.
\end{equation}
Clearly $g(\bDelta)$ is minimized at $\bDelta = \bm{0}$ by the definition of $\bOmega^*$, so the first derivative of $g(\bDelta)$ vanishes at $\bm{0}$, provided that it is differentiable at $\bm{0}$. Define the matrix $\bm{H}_{\bm{B}} = [h_{\bm{B}}\{(i,j),(l,m)\}]$, where
\begin{equation}
h_{\bm{B}}\{(i,j),(l,m)\}= \tr\left\{\bm{B}^{-1}\bm{E}_{(i,j)}\bm{B}^{-1}\bm{E}_{(l,m)}\right\}.
\end{equation}
Using standard matrix calculus (for example, see Section 15.9 of \cite{harville2008matrix}), we can find that the Hessian of $g(\bDelta)$ is the $\#\mathcal{V}_{\bGamma} \times \#\mathcal{V}_{\bGamma}$ matrix $\bm{H}_{\bDelta+\bOmega^*}$, whose $\{(i,j),(l,m)\}$th entry for $(i,j), (l,m) \in \mathcal{V}_{\bGamma}$ is given by
\begin{equation}
\frac{\partial^2 g(\bDelta)}{\partial u_{ij} \partial u_{lm}} = \mathrm{tr}\left\{(\bDelta + \bOmega^*)^{-1}\bm{E}_{(i,j)}(\bDelta + \bOmega^*)^{-1}\bm{E}_{(l,m)}\right\}.
\end{equation}
Thus the Laplace approximation $p^*\{\bGamma \mid \bX^{(n)}\}$ to the posterior probability $p\{\bGamma \mid \bX^{(n)}\}$ is given by
\begin{equation}
\begin{split}
\label{eqn:postprbapprox}
p^*\{\bGamma|\bX^{(n)}\} &\propto  C_{\bGamma}\exp\{-n\,h(\bOmega^*)/2\} \left\{\mathrm{det}(\bOmega^*)\right\}^{-n/2} \exp\{-n\,g(\bm{0})/2\} \\
&\qquad \times (2\pi)^{\#\mathcal{V}_{\bGamma}/2}(n/2)^{-\#\mathcal{V}_{\bGamma}/2}\left[\mathrm{det}\left\{\left.\frac{\partial^2 g(\bDelta)}{\partial \bDelta \partial \bDelta^T}\right|_{\bm{0}}\right\}\right]^{-1/2} \\
&= C_{\bGamma}\exp\{-n\,h(\bOmega^*)/2\} (\pi/n)^{\#\mathcal{V}_{\bGamma}/2} \{\mathrm{det}(\bm{H}_{\bOmega^*})\}^{-1/2}.
\end{split}
\end{equation}
The approximation in (\ref{eqn:postprbapprox}) is meaningful only if all the graphical lasso estimates of the off-diagonal elements corresponding to the graph generated by $\bGamma$ are non-zero; otherwise the derivative of $g(\bDelta)$ does not exist. A similar situation arises in the context of regression models; see \cite{yuan2005efficient} and \cite{mckay2013fast}. In the next section, we show that such ``non-regular models" can essentially be ignored for the purpose of posterior probability evaluation.

\subsection{Ignorability of non-regular models}
\label{subsec:nonregular}
As discussed in the previous section, the objective function of the graphical lasso problem is not differentiable if the graphical lasso solution is zero for at least one pair $(i,j) \in E$. These models are referred to as non-regular models. This essentially means that given a fixed graphical structure index $\bGamma$, the graphical lasso solution is $\omega_{ij}^* = 0$ for at least one $\gamma_{ij}=1$. Let us assume, for notational simplicity, that the first $t$ elements of $\bGamma$ are 1 and the rest are 0. Also, among those $t$ 1's, the last $r$ of them have corresponding graphical lasso solution equal to zero. For such a non-regular model, we argue that the submodel $\bGamma'$, with first $(t-r)$ 1's and rest 0's, provides the same graphical lasso solution for the non-zero elements as the bigger model $\bGamma$. This means that for $(i,j)$ such that $\gamma_{ij} = \gamma'_{ij} = 1$, the graphical lasso solution corresponding to $\bGamma$, given by $\omega_{\bGamma,ij}^*$ is identical with that corresponding to $\bGamma'$, given by $\omega_{\bGamma',ij}^*$. We refer to such a submodel $\bGamma'$ as the regular submodel of the non-regular model $\bGamma$.
\begin{lemma}
\label{lemma:nonregular}
For a submodel $\bGamma'$ of $\bGamma$ as defined above, the graphical lasso solution corresponding to the two models are identical.
\end{lemma}
We give a proof of the above lemma in the appendix. For notational convenience, let us denote the precision matrix $\bOmega$ corresponding to the structure indicator $\bGamma$ by $\bOmega_{\bGamma}$, and corresponding matrix $\bDelta_{\bGamma}$ is defined by $\bOmega_{\bGamma} - \bOmega^* = (\!(u_{\bGamma,ij})\!)$.
We denote the graphical lasso solution in the non-regular model $\bGamma$ and the corresponding regular submodel $\bGamma'$ by $\bOmega^*$. The ratio of the posterior model probabilities of the two model is given by,
\begin{equation}
\label{eqn:nonregularratio}
\frac{p\{\bGamma|\bX^{(n)}\}}{p\{\bGamma'|\bX^{(n)}\}} = \frac{C_{\bGamma} \int_{\bDelta_{\bGamma}+\bOmega^* \in \mathcal{M}^+} \exp\{-n\,h(\bDelta_{\bGamma})/2\} \prod_{(i,j) \in \mathcal{V}_{\bGamma}}du_{\bGamma,ij}}{C_{\bGamma'} \int_{\bDelta_{\bGamma'}+\bOmega^* \in \mathcal{M}^+} \exp\{-n\,h(\bDelta_{\bGamma'})/2\} \prod_{(i,j) \in \mathcal{V}_{\bGamma'}}du_{\bGamma',ij}}.
\end{equation}
The following result shows the ignorability of the non-regular models.

\begin{theorem}
\label{theorem:nonregular}
Consider the prior on $\bGamma$ as given in (\ref{prior:gamma1}) or (\ref{prior:gamma2}) with $q < 1/2$. The posterior probability of a non-regular model $\bGamma$, as defined above, is always less than that of its regular submodel $\bGamma'$.
\end{theorem}
\begin{proof}
Using (\ref{eqn:postratio}), we have,
\begin{equation}
\frac{p\{\bGamma|\bX^{(n)}\}}{p\{\bGamma'|\bX^{(n)}\}}  = 
\frac{C_{\bGamma}\int_{\|\bDelta_{\bGamma}\|_2 \leq \epsilon_n} \exp\{-n\,h(\bDelta_{\bGamma})/2\}\prod_{(i,j) \in \mathcal{V}_{\bGamma}}du_{\bGamma,ij} + o(1)}{C_{\bGamma'} \int_{\|\bDelta_{\bGamma'}\|_2 \leq \epsilon_n} \exp\{-n\,h(\bDelta_{\bGamma'})/2\} \prod_{(i,j) \in \mathcal{V}_{\bGamma'}}du_{\bGamma',ij} + o(1)}. \nonumber \\
\end{equation}
Now, note that for $(i,j)$ such that $\gamma_{ij} = \gamma'_{ij} =1$, we have,
\begin{equation*}
 \{u_{\bGamma,ij}: \|\bDelta_{\bGamma}\|_2 \leq \epsilon_n\} \subset \{u_{\bGamma',ij}: \|\bDelta_{\bGamma'}\|_2 \leq \epsilon_n\}. 
\end{equation*} 
 Hence, using Lemma \ref{lemma:Nonregularitylemma}, we get
\begin{eqnarray}
\frac{p\{\bGamma|\bX^{(n)}\}}{p\{\bGamma'|\bX^{(n)}\}} & \leq & 
 \frac{C_{\bGamma}}{C_{\bGamma'}}\int_{\|\bDelta_{\bGamma}\|_2 \leq \epsilon_n} \exp\left(-\frac{n}{2}\frac{2\lambda}{n}\sum_{\gamma_{ij} = 1,\gamma'_{ij}=0}|u_{\bGamma,ij}|\right)\prod_{(i,j) \in \mathcal{V}_{\bGamma}\cap \mathcal{V}_{\bGamma'}^c}du_{\bGamma,ij} \nonumber \\
& \leq & \frac{C_{\bGamma}}{C_{\bGamma'}} \int \exp\left(-\lambda \sum_{\gamma_{ij} = 1,\gamma'_{ij}=0}|u_{\bGamma,ij}|\right)\prod_{(i,j) \in \mathcal{V}_{\bGamma}\cap \mathcal{V}_{\bGamma'}^c}du_{\bGamma,ij} \nonumber \\
& = & \frac{C_{\bGamma}}{C_{\bGamma'}} \left( \frac{2}{\lambda} \right)^{\#\bGamma - \#\bGamma'} \nonumber \\
& = & \left(\frac{q}{1-q}\right)^r \frac{\beta(\bGamma)}{\beta(\bGamma')} \nonumber \\
& \leq & \left(\frac{q}{1-q}\right)^r.
\end{eqnarray}
The last inequality follows from the fact that if the prior as in (\ref{prior:gamma1}) is used, then $\P(\bar{R} \geq \#\bGamma) \leq \P(\bar{R} \geq \#\bGamma')$ since $\#\bGamma > \#\bGamma'$. For the other prior as in (\ref{prior:gamma2}), the inequality follows trivially as it involves the ratio of two indicator variables only.
\par For $q < 1/2$, the above ratio is less than $1$. This completes the proof.
\end{proof}
The above result is particularly important in the sense that we can focus on the regular models only, ignoring the non-regular ones especially if $q$ is chosen to be small. While approximating the posterior probabilities of the regular models, we re-normalize the values considering the regular models only.

\subsection{Error in Laplace approximation}
The approximation in the posterior probability of the graphical model is based on a Taylor series expansion of the function $h(\bOmega)$ around the graphical lasso solution $\bOmega^*$. Let $\bDelta = \bOmega - \bOmega^*$, and $\mathrm{vec}(\bDelta)$ denote the vectorized version of $\bDelta$, but excluding entries corresponding to the missing edges in the underlying graphical model. Thus $\mathrm{vec}(\bDelta)$ is a vector of dimension $\#\mathcal{V}_{\bGamma}$ corresponding to the graphical structure indicator $\bGamma$. If the graphical model is $s$-sparse, that is, there are $s$ edges present in the graph, then $\#\mathcal{V}_{\bGamma} = p+s$. The following result gives the bound on the remainder term of the Taylor series expansion under the above assumptions.
\begin{lemma}
\label{lemma:Rnbound}
Consider a graphical model with $p$ variables such that the graph is $s$-sparse. Then, with probability tending to 1, the remainder term in the expansion of the function $h(\bOmega)$ as defined in (\ref{eqn:hfunc}), around the graphical lasso solution $\bOmega^*$, is bounded by $(p+s)\|\bDelta\|_2^2\left(C_1\|\bDelta\|_2 + C_2 \|\bDelta\|_2^2\right)/2$, where $\bDelta = \bOmega - \bOmega^*$.
\end{lemma}
This result can be used to find a bound for the error in Laplace approximation of the posterior probabilities of the graphical model structures. The following result gives the condition for which the error in approximation is asymptotically negligible.
\begin{theorem}
\label{theorem:appxerror}
The error in Laplace approximation of the posterior probability of a graphical model structure is asymptotically negligible if $(p+s)^2\epsilon_n \rightarrow 0$, where $\epsilon_n$ is the posterior convergence rate, that is, the error in the Laplace approximation tends to zero if $n^{-1/2}(p+s)^{5/2}(\log p)^{1/2} \rightarrow 0$.
\end{theorem}
The proof of the above result depends on several additional results, including Lemma \ref{lemma:Rnbound} involving the bound on the remainder term in the Taylor series expansion of $h(\bOmega)$. We give a proof of the above result along with these additional results in the appendix.

\section{Simulation results}

We perform a simulation study to assess the performance of the Bayesian method for graphical structure learning. We use 4 different models for our simulations, and we specify these models in terms of the elements of the covariance matrix $\bSigma = (\!(\sigma_{ij})\!)$ or the precision matrix $\bOmega = (\!(\omega_{ij})\!)$, as follows:
\begin{enumerate}
\item Model 1: AR(1) model, $\sigma_{ij} = 0.7^{|i-j|}$.
\item Model 2: AR(2) model, $\omega_{ii} = 1, \omega_{i,i-1} = \omega_{i-1,i} = 0.5, \omega_{i,i-2}=\omega_{i-2,i}=0.25$.
\item Model 3: Star model, where every node is connected to the first node, and $\omega_{ii} = 1, \omega_{1,i} = \omega_{i,1} = 0.1$, and $\omega_{ij}=0$ otherwise.
\item Model 4: Circle model, $\omega_{ii}=2, \omega_{i,i-1}=\omega_{i,i-1}=1, \omega_{1,p}=\omega_{p,1}=0.9$.
\end{enumerate}

Corresponding to each model, we generate samples of size $n = 100, 200$ and dimension $p = 30, 50, 100$. The penalty parameter for the graphical lasso algorithm is chosen to be 0.5 and the value of $q$ appearing in the prior of the graphical structure indicator to be 0.4. We run 100 replications for each of the models and find the median probability model for each replication. To assess the performance of the median probability model (denoted by `MPP'), we compute the specificity, sensitivity and Matthews Correlation Coefficient (MCC) averaged across the replications as defined below and also compute the same for the graphical lasso (denoted by `GL'). The results are presented in Table \ref{Ch4table:simu}.
\begin{eqnarray}
\mathrm{SP} &=&  \frac{\mathrm{TN}}{\mathrm{TN}+\mathrm{FP}},\, \mathrm{SE} =  \frac{\mathrm{TP}}{\mathrm{TP}+\mathrm{FN}} \nonumber \\
\mathrm{MCC} &=&  \frac{\mathrm{TP}\times \mathrm{TN} - \mathrm{FP}\times \mathrm{FN}}{\sqrt{(\mathrm{TP}+\mathrm{FP})(\mathrm{TP}+\mathrm{FN})(\mathrm{TN}+\mathrm{FP})(\mathrm{TN}+\mathrm{FN})}},
\end{eqnarray}
where TP, TN, FP and FN respectively denote the true positives, true negatives, false positives and false negatives in the selected model, which in our case is the median probability model.

\begin{sidewaystable}
\footnotesize 
\caption{Simulation results for different structures of precision matrices \label{Ch4table:simu}}
\begin{tabular}{ccccccccccccccccc}

  \hline
  & & \multicolumn{7}{c}{$n = 100$} &  & \multicolumn{7}{c}{$n = 200$} \\
    \cline{3-9} \cline {11-17}
    & & \multicolumn{3}{c}{MPP} & & \multicolumn{3}{c}{GL} & &  \multicolumn{3}{c}{MPP} & & \multicolumn{3}{c}{GL} \\
         \cline{3-5} \cline{7-9} \cline{11-13} \cline {15-17}
     Model &$p$ & SP & SE & MCC & & SP & SE & MCC & & SP & SE & MCC & & SP & SE & MCC \\
    \hline
&&&&&&&&&&&&&&&&\\
    & 30 & 0.977 & 0.941 & 0.831 & & 0.961 & 0.983 & 0.784 & & 0.986 & 0.996 & 0.907 & & 0.969 & 1.000 & 0.823 \\
    &    &(0.003)&(0.019)&(0.015)& &(0.003)&(0.010)&(0.013)& &(0.002)&(0.003)&(0.014)& &(0.002)&(0.000)&(0.013)\\
  AR(1)  & 50 & 0.987 & 0.953 & 0.841 & & 0.977 & 0.986 & 0.785 & & 0.991 & 0.992 & 0.903 & & 0.980 & 1.000 & 0.823\\
    &    &(0.002)&(0.013)&(0.010)& &(0.001)&(0.004)&(0.010)& &(0.001)&(0.004)&(0.008)& &(0.001)&(0.000)&(0.006) \\
    & 100& 0.992 & 0.967 & 0.837 & & 0.989 & 0.991 & 0.804 & & 0.994 & 0.995 & 0.890 & & 0.991 & 0.999 & 0.827 \\
    &    &(0.001)&(0.008)&(0.007)& &(0.001)&(0.003)&(0.006)& &(0.001)&(0.002)&(0.008)& &(0.001)&(0.001)&(0.006)\\
&&&&&&&&&&&&&&&&\\
    & 30 & 0.975 & 0.470 & 0.546 & & 0.964 & 0.535 & 0.558 & & 0.987 & 0.495 & 0.617 & & 0.982 & 0.517 & 0.610\\
    &    &(0.003)&(0.014)&(0.013)& &(0.002)&(0.013)&(0.012)& &(0.002)&(0.008)&(0.008)& &(0.002)&(0.009)&(0.007)\\
    AR(2)& 50 & 0.983 & 0.462 & 0.541 & & 0.971 & 0.508 & 0.522 & & 0.993 & 0.489 & 0.629 & & 0.987 & 0.534 & 0.622\\
    &    &(0.001)&(0.013)&(0.011)& &(0.002)&(0.010)&(0.009)& &(0.001)&(0.005)&(0.007)& &(0.001)&(0.001)&(0.006)\\
    & 100& 0.989 & 0.470 & 0.537 & & 0.980 & 0.531 & 0.514 & & 0.995 & 0.484 & 0.624 & & 0.993 & 0.529 & 0.624\\
    &    &(0.001)&(0.006)&(0.006)& &(0.001)&(0.007)&(0.007)& &(0.001)&(0.006)&(0.004)& &(0.001)&(0.009)&(0.005)\\
&&&&&&&&&&&&&&&&\\  
    & 30 & 0.947 & 0.289 & 0.228 & & 0.937 & 0.310 & 0.224 & & 0.995 & 0.210 & 0.378 & & 0.993 & 0.252 & 0.402\\
    &    &(0.004)&(0.038)&(0.036)& &(0.003)&(0.043)&(0.036)& &(0.001)&(0.032)&(0.041)& &(0.001)&(0.036)&(0.038)\\
    Star& 50 & 0.945 & 0.492 & 0.332 & & 0.934 & 0.514 & 0.317 & & 0.993 & 0.475 & 0.585 & & 0.990 & 0.514 & 0.577\\
    &    &(0.003)&(0.034)&(0.025)& &(0.003)&(0.035)&(0.023)& &(0.000)&(0.034)&(0.024)& &(0.001)&(0.032)&(0.022)\\
    & 100& 0.939 & 1.000 & 0.485 & & 0.927 & 1.000 & 0.452 & & 0.988 & 1.000 & 0.792 & & 0.984 & 1.000 & 0.748\\
    &    &(0.002)&(0.000)&(0.007)& &(0.002)&(0.000)&(0.005)& &(0.000)&(0.000)&(0.008)& &(0.001)&(0.000)&(0.007)\\
&&&&&&&&&&&&&&&&\\
    & 30 & 0.733 & 1.000 & 0.399 & & 0.694 & 1.000 & 0.369 & & 0.719 & 1.000 & 0.388 & & 0.674 & 1.000 & 0.354\\
    &    &(0.004)&(0.000)&(0.003)& &(0.006)&(0.000)&(0.004)& &(0.005)&(0.000)&(0.004)& &(0.004)&(0.000)&(0.003)\\
    Circle& 50 & 0.831 & 1.000 & 0.409 & & 0.822 & 1.000 & 0.398 & & 0.833 & 1.000 & 0.411 & & 0.814 & 1.000 & 0.390\\
    &    &(0.003)&(0.000)&(0.003)& &(0.002)&(0.000)&(0.003)& &(0.002)&(0.000)&(0.003)& &(0.002)&(0.000)&(0.002)\\
    & 100& 0.891 & 1.000 & 0.378 & & 0.894 & 1.000 & 0.383 & & 0.903 & 1.000 & 0.399 & & 0.902 & 1.000 & 0.397\\
    &    &(0.001)&(0.000)&(0.002)& &(0.001)&(0.000)&(0.002)& &(0.008)&(0.000)&(0.002)& &(0.001)&(0.000)&(0.002)\\
&&&&&&&&&&&&&&&&\\
    \hline
  \end{tabular}
\end{sidewaystable}

\section{Illustration with real data}
In this section we illustrate the Bayesian graphical structure learning method with the stock price data from Yahoo! Finance. Description of the data set can be found in \cite{liu2009nonparanormal} and available in the \texttt{huge} package on \texttt{CRAN} \citep{zhao2012huge} as \texttt{stockdata}. The data set consists of closing prices of stocks that were consistently included in the S\&P 500 index in the time period January 1, 2003 to January 1, 2008 for a total of 1258 days. The stocks are also categorized into 10 Global Industry Classification Standard (GICS) sectors, namely, ``Health Care", ``Materials", ``Industrials", ``Consumer Staples", ``Consumer Discretionary", ``Utilities", ``Information Technology", ``Financials", ``Energy", ``Telecommunication Services".

\par Denoting $Y_{tj}$ as the closing stock price for the $j$th stock on day $t$, we construct the $1257 \times 452$ data matrix $\bm{S}$ with entries $s_{tj} = \log(Y_{(t+1)j}/Y_{tj}),\, t = 1,\ldots,1257,\, j = 1,\ldots,452$. For analysis, we construct the data matrix $\bX$ by standardizing $\bm{S}$, so that each stock has mean zero and standard deviation one. We find the median probability model as selected by the Bayesian graphical structure learning method. The corresponding graphical structure is displayed in Figure \ref{graph1}. The vertices of the graph are colored corresponding to the different GICS sectors. We find that stocks from the same sectors tend to be related with other members from that category, and generally not related across different sectors, though there are some connections, which may be due to some other possible latent factors affecting all of them. The grouping of the stocks corresponding to their sectors is expected, implying that the stock prices for a particular sector are conditionally independent of those of other sectors.

\par We also individually study data pertaining to some of the specific sectors to have a closer look at the strength of the groupings where perturbations due to latent factors is least expected. For this, we consider the sectors ``Utilities" and ``Information Technology". The graphical structure is displayed in Figure \ref{graph2}. The stock prices for the two sectors clearly separate as desired.

\begin{figure}[h]
\label{graph1}
\centering
\includegraphics[width=4.5in, height = 4in]{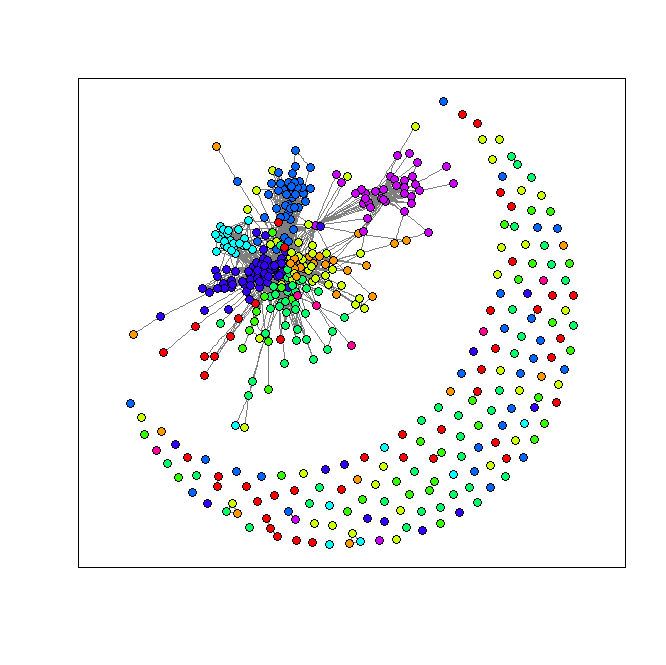}
\caption{Graphical structure of the median probability model selected by the Bayesian graphical structure learning method.}
\label{fig:graph1}
\end{figure}

\begin{figure}[h]
\label{graph2}
\centering
\includegraphics[width=4.5in, height = 4in]{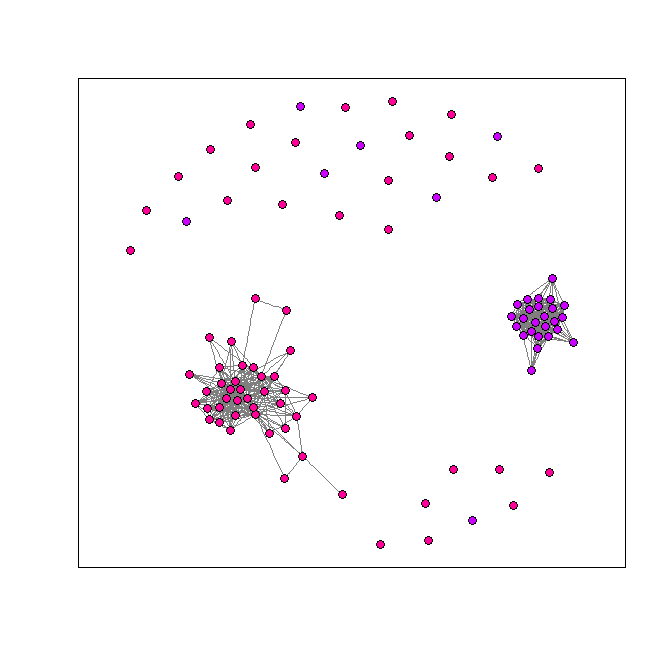}
\caption{Graphical structure corresponding to the subgraph corresponding to the sectors ``Utilities" [red] and ``Information Technology"[violet].}
\label{fig:graph2}
\end{figure}

\appendix

\section{Proofs} \label{app}

We now give a proof of the result on posterior convergence rate of the precision matrix.
\begin{proof}[Proof of Theorem \ref{theorem:postconvrate}]

In order to establish the rates of convergence of the posterior distribution, we first need to check the prior concentration rate, that is,
\begin{equation}
\Pi\left\{B(p_{\bOmega_0},\epsilon_n)\right\} := \Pi\left\{p: K(p_{\bOmega_0},p_{\bOmega}) \leq \epsilon_n^2, V(p_{\bOmega_0},p_{\bOmega}) \leq \epsilon_n^2\right\} \geq \exp(-bn\epsilon_n^2),
\end{equation}
where $K(p_{\bOmega_0},p_{\bOmega}) = \int p_{\bOmega_0}\log\left(p_{\bOmega_0}/p_{\bOmega}\right), \, V(p_{\bOmega_0},p_{\bOmega}) = \int p_{\bOmega_0}\left\{\log \left(p_{\bOmega_0}/p_{\bOmega}\right)\right\}^2$.
Note that, for $\bm{Z} \sim \mathrm{N}_p(\bm{0},\bSigma)$ and a $p \times p$ symmetric matrix $\bm{A}$, we have,
\begin{equation}
\label{eqn:quadform}
\E(\bm{Z}^T\bm{A}\bm{Z}) = \tr(\bm{A}\bSigma),\, \quad \mathrm{Var}(\bm{Z}^T\bm{A}\bm{Z}) = 2 \,\tr(\bm{A}\bSigma\bm{A}\bSigma).
\end{equation}
We use the above result to find the expressions for $K(p_{\bOmega_0},p_{\bOmega})$ and $V(p_{\bOmega_0},p_{\bOmega})$. Denoting the eigenvalues of the matrix $\bOmega_0^{-1/2}\bOmega\bOmega_0^{-1/2}$ by $d_{i},\, i=1,\ldots,p$, using $\tr(\bmA\bm{B}) = \tr(\bm{B}\bmA)$ and (\ref{eqn:quadform}), we get,
\begin{eqnarray}
K(p_{\bOmega_0},p_{\bOmega}) &=& \frac{1}{2}(\log \det \bOmega_0 - \log \det \bOmega) - \frac{1}{2}\tr(\bmI_p - \bOmega\bOmega_0^{-1}) \nonumber \\
&=& \frac{1}{2}(\log \det \bOmega_0 - \log \det \bOmega) - \frac{1}{2}\tr(\bmI_p - \bOmega_0^{-1/2}\bOmega\bOmega_0^{-1/2}) \nonumber \\
&=& -\frac{1}{2}\sum_{i=1}^p\log d_{i} - \frac{1}{2}\sum_{i=1}^{p}(1-d_{i}).
\end{eqnarray}
Now, $K(p_{\bOmega_0},p_{\bOmega}) \geq h^2(p_{\bOmega_0},p_{\bOmega})$, and from an argument in Lemma \ref{lemma:HellingervsFrob} it implies that if $K(p_{\bOmega_0},p_{\bOmega}) \leq \epsilon_n^2$, then $\mathop{\mathrm{max}}_i|d_{i} - 1| < 1$. Hence we can expand $\log d_i$ in the powers of $(1-d_i)$ to get
\begin{equation}
K(p_{\bOmega_0},p_{\bOmega}) \sim \frac{1}{4}\sum_{i=1}^p (1-d_{i})^2.
\end{equation}
Also, 
\begin{eqnarray}
V(p_{\bOmega_0},p_{\bOmega}) &=& \frac{1}{2}\tr(\bmI_p - 2\bOmega\bOmega_0^{-1}+ \bOmega\bOmega_0^{-1}\bOmega\bOmega_0^{-1}) \nonumber \\
&=& \tr(\bmI_p - 2\bOmega_0^{-1/2}\bOmega\bOmega_0^{-1/2} + \bOmega_0^{-1/2}\bOmega\bOmega_0^{-1}\bOmega\bOmega_0^{-1/2}) \nonumber \\
&=& \frac{1}{2}\tr(\bmI_p - \bOmega_0^{-1/2}\bOmega\bOmega_0^{-1/2})^2 \nonumber \\ 
&=& \frac{1}{2}\sum_{i=1}^p(1-d_{i})^2. \nonumber 
\end{eqnarray}
Thus,
\begin{equation}
\label{eqn:priorconc1}
\Pi\left\{p: K(p_{\bOmega_0},p_{\bOmega}) \leq \epsilon_n^2, V(p_{\bOmega_0},p_{\bOmega}) \leq \epsilon_n^2\right\}  \geq  \Pi\left\{\sum_{i=1}^p(1-d_{i})^2 \leq 8\epsilon_n^2\right\}.
\end{equation}
Now, using the assumptions on the true precision matrix $\bOmega_0$ and the matrix norm relations given by (\ref{eqn:matrixnorm}), we have,
\begin{eqnarray}
\label{eqn:priorconc2}
\sum_{i=1}^{p}(1-d_i)^2 &=& \|\bmI_p - \bOmega_0^{-1/2}\bOmega\bOmega_0^{-1/2}\|_2^2 \nonumber \\
&=& \|\bOmega_0^{-1/2}(\bOmega_0 - \bOmega)\bOmega_0^{-1/2}\|_2^2 \nonumber \\
& \leq & \|\bOmega_0^{-1}\|_{(2,2)}^2\|\bOmega_0 - \bOmega\|_2^2 \nonumber \\
& \leq & \varepsilon_0^{-2}\|\bOmega_0 - \bOmega\|_2^2.
\end{eqnarray}
Hence, equations (\ref{eqn:priorconc1}) and (\ref{eqn:priorconc2}), along with (\ref{eqn:matrixnorm}) give,
\begin{eqnarray}
\Pi\left\{p: K(p_{\bOmega_0},p_{\bOmega}) \leq \epsilon_n^2, V(p_{\bOmega_0},p_{\bOmega}) \leq  \epsilon_n^2\right\} &\geq& \Pi\left\{\|\bOmega_0 - \bOmega\|_2^2 \leq c\epsilon_n^2\right\} \nonumber \\
&\geq & \Pi\left(\|\bOmega_0 - \bOmega\|_{\infty} \leq c'\epsilon_n/p\right). \nonumber
\end{eqnarray}
The components of $\bOmega$ are not independently distributed, but a truncation applies because of the positive definite restriction. However, as the true $\bOmega_0$ lies in the set of positive definite matrices which is open, the truncation can only increase concentration in a small ball centered at the truth, so we can pretend componentwise independence for the purpose of lower bounding the above prior probability. This gives
\begin{equation}
\Pi\left(\|\bOmega_0 - \bOmega\|_{\infty} \leq c'\epsilon_n/p\right)  \gtrsim  \left(c'\epsilon_n/p\right)^{p+s}.
\end{equation}

The prior concentration rate condition thus gives,
\begin{equation}
(p+s)(\log p + \log \frac{1}{\epsilon_n}) \asymp n\epsilon_n^2,
\end{equation}
so as to get $\epsilon_n  = n^{-1/2}(p+s)^{1/2}(\log n)^{1/2}.$

Next, we need to construct tests for $H_0 :\bOmega = \bOmega_0$ against the alternative $H_1 : \|\bOmega_0 - \bOmega\|_2 \geq \epsilon_n$. Let $p_{\bOmega,n}(X_1,\ldots,X_n)$ denote the joint density of the observations and let $\Pi$ be the prior. From the results of \cite{birge1984} and \cite{lecamasymptotic}, we know that for a probability measure $P_0$ and any convex set $\mathcal{P}$ of probability measures, there exists tests $\phi_n$ such that
\begin{equation}
\label{eqn:birgelecam}
P_0^n \phi_n \leq e^{-nh^2(P_0, \mathcal{P})/2}, \mathop{\mathrm{sup}}_{P \in \mathcal{P}} P^n (1 - \phi_n)  \leq e^{-nh^2(P_0, \mathcal{P})/2},
\end{equation}
where $h^2(P_0, \mathcal{P}) = \mathrm{min}\{h^2(P_0,P):P \in \mathcal{P}\}.$

Let $\mathcal{P}_2 = \{p_{\bOmega_2}:\|\bOmega_2 - \bOmega_1\|_2 \leq c_0^{-1/2}\epsilon_n/2\}$, where $c_0$ is the constant appearing in Lemma \ref{lemma:HellingervsFrob} and $h(p_{\bOmega_0},p_{\bOmega_1}) \geq \epsilon_n$. We claim that $h(p_{\bOmega},p_{\bOmega_0}) > \epsilon_n/2$ for any $p_{\bOmega}$ in the convex hull $\mathrm{conv}(\mathcal{P}_2)$ of $\mathcal{P}_2$. To see this, represent $p_{\bOmega}$ as
\begin{equation}
p_{\bOmega} = \int_{\bOmega_2:\|\bOmega_2 - \bOmega_1\|_2 \leq c_0^{-1/2}\epsilon_n/2} p_{\bOmega_2}\,d\Phi(\bOmega_2),
\end{equation}
where $\Phi$ is an arbitrary probability measure on $c_0^{-1/2}\epsilon_n/2$-ball around $\bOmega_1$ in terms of Frobenius distance.
Then, for any $ p_{\bOmega} \in \mathrm{conv}(\mathcal{P}_2)$, by Lemma \ref{lemma:HellingervsFrob} and the convexity of the squared Hellinger distance,
\begin{equation}
h^2(p_{\bOmega},p_{\bOmega_1}) \leq  \int_{\bOmega_2:\|\bOmega_2 - \bOmega_1\|_2 \leq c_0^{-1/2}\epsilon_n/2} h^2(p_{\bOmega_1},p_{\bOmega_2})\,d\Phi(\bOmega_2) < \epsilon_n^2/4.
\end{equation}
This implies that $h\{p_{\bOmega_0}, \mathrm{conv}(\mathcal{P}_2)\} > \epsilon_n/2$ by the triangle inequality. Thus, by (\ref{eqn:birgelecam}), we can find tests for $\bOmega=\bOmega_0$ vs. $\bOmega \in \mathcal{P}_2$ such that the error probabilities are bounded by $\exp(-n\epsilon_n^2/8)$.

In order to get a test for $H_0$ vs. $H_1$ with similar error probability, we also need to cover the alternative with balls of size $\epsilon_n/2$ and satisfy the metric entropy condition
\begin{equation}
\log N(\epsilon_n/2, \mathcal{P}_n,\|\cdot \|'_2) \leq c_1n\epsilon_n^2,
\end{equation}
where $\|\cdot\|'_2$ is the distance on $p_{\bOmega}$ induced by $\|\cdot\|_2$ on $\bOmega$, $c_1 > 0$ is a constant and $\mathcal{P}_n \subset \mathcal{P}$ is a suitable subset of $\mathcal{P}$, called a sieve, such that $\Pi(\mathcal{P}_n^c)$ is exponentially small. For a graph $G$ with $p$ vertices, consider the sieve $\mathcal{P}_n$ to be the space of all densities $p_{\bOmega}$ such that the graph corresponding to $\bOmega$ has maximum number of edges $\bar{r} < \binom{p}{2}/2$ and each off-diagonal entry of $\bOmega$ is at most $L$. Then the metric entropy condition is given by
\begin{equation}
\label{eqn:metricentropy}
\log \left\{\sum_{j=1}^{\bar{r}}\left(\frac{L}{\epsilon_n}\right)^j\binom{\binom{p}{2}}{j}\right\} \leq \log \left\{\bar{r}\left(\frac{L}{\epsilon_n}\right)^{\bar{r}}\binom{\binom{p}{2}}{\bar{r}}\right\},
\end{equation}
where we choose $L \in [b_2n\epsilon_n^2,b_2n\epsilon_n^2+1]$ to ensure that $\binom{p}{2}\exp(-L) \leq \exp(-b_3n\epsilon_n^2)$, for some constants $b_2$ and $b_3$, and that $b_3$ can be made as large as we want by making $b_2$ larger. Thus the best solution of (\ref{eqn:metricentropy}) leads to the relation
\begin{equation}
\label{eqn:barrrate}
\log \bar{r} + \bar{r}\log p + \bar{r} \log (\frac{1}{\epsilon_n}) + \bar{r} \log (n\epsilon_n^2) \asymp n\epsilon_n^2,
\end{equation}
which is satisfied if we choose $\bar{r} = b_1n\epsilon_n^2/\log n$, $b_1$ large. Also, for this choice of $\bar{r}$, we have,
\begin{equation}
\label{eqn:Rbound}
\P(\bar{R} > \bar{r}) \leq \exp(-a_2'b_1n\epsilon_n^2),
\end{equation}
where $a_2'b_1$ can be made as large as possible by making $b_1$ large. 
For the bound on the prior probability of the complement $\mathcal{P}_n^c$ of the above sieve, we have, using the condition on prior (\ref{prior:gamma1}),
\begin{equation}
\label{eqn:sievecomplement}
\Pi(\mathcal{P}_n^c) \leq \P(\bar{R} > \bar{r}) + \exp(-b_3n\epsilon_n^2).
\end{equation}
For the prior (\ref{prior:gamma2}), the first term in (\ref{eqn:sievecomplement}) is exactly zero, and for the prior (\ref{prior:gamma1}), from equation (\ref{eqn:Rbound}),
\begin{equation}
\Pi(\mathcal{P}_n^c) \leq \exp(-c_3n\epsilon_n^2),
\end{equation}
where $c_3$ is a constant which can be made as large as we please by making $b_1, b_3$ larger. Note that under the condition $n\epsilon_n^2/\log n \ll \binom{p}{2}$, the requirement $\bar{r} < \binom{p}{2}/2$ is satisfied as $n \rightarrow \infty$. Hence $\epsilon_n$ as found above is the desired posterior convergence rate.
\end{proof}

The following lemma establishes a norm equivalence necessary for finding posterior convergence rate and metric entropy calculations.

\begin{lemma}
\label{lemma:HellingervsFrob}
If $p_{\bOmega_k}$ is the density of $\mathrm{N}_p(\bm{0},\bOmega_k^{-1}),\, k = 1,2$, then for all $\bOmega_k \in \mathcal{U}(\varepsilon_0,s), \, k = 1,2$,
\begin{enumerate}[(i)]
\item $c_0^{-1}\|\bOmega_1 - \bOmega_2\|_2^2 \leq h^2(p_{\bOmega_1},p_{\bOmega_2})$,\, when\, $\|\bOmega_1 - \bOmega_2\|_2 < \varepsilon_0$,
\item $h^2(p_1,p_2) \leq c_0\|\bOmega_1 - \bOmega_2\|_2^2$,
\end{enumerate}
for some universal constant $c_0 > 0$.
\end{lemma}
\begin{proof}
Let $d_{i},\, i=1,\ldots,p$ be the eigenvalues of the matrix $\bmA = \bOmega_1^{-1/2}\bOmega_2\bOmega_1^{-1/2}$. Half squared Hellinger distance between $p_1$ and $p_2$ is given by
\begin{equation}
\label{eqn:hellingerp1p2}
1- \frac{\left\{\det(\bmA)\right\}^{-1/4}}{\left(\det\left[\frac{1}{2}\{\bmI + \bmA^{-1}\}\right]\right)^{1/2}} = 1 - \frac{\prod_{i=1}^p d_{i}^{-1/4}}{\{\prod_{i=1}^p\frac{1}{2}(1 + d_{i}^{-1})\}^{1/2}},
\end{equation}
and the Frobenius norm of the difference between $\bOmega_1$ and $\bOmega_2$ is given by, from (\ref{eqn:matrixnorm}),
\begin{eqnarray}
\label{eqn:Frobp1p2}
\|\bOmega_1 - \bOmega_2\|_2^2 &=& \|\bOmega_1^{1/2}(\bmI_p - \bmA)\bOmega_1^{1/2}\|_2^2 \nonumber \\
&\leq & \|\bOmega_1\|_{(2,2)}^2\|\bmI_p - \bmA\|_2^2\nonumber \\
&=& \|\bOmega_1\|_{(2,2)}^2 \mathrm{tr}\left(\bmI_p - \bmA\right)^2 \nonumber \\
&=& \|\bOmega_1\|_{(2,2)}^2 \sum_{i=1}^p(d_{i}-1)^2 \nonumber \\
&\leq & \varepsilon_0^{-2} \sum_{i=1}^p(d_{i}-1)^2.
\label{eqn:Frobbound}
\end{eqnarray}

First we show that either $\|\bOmega_1 - \bOmega_2\|_2^2 \leq \delta^2$ or $h^2(p_{\bOmega_1},p_{\bOmega_2}) \leq 2\delta^2$ implies $|d_{i} - 1| < 1$ for all $i = 1,\ldots, p$ for sufficiently small $\delta$. This is necessary to expand $d_i$ in powers of $(1-d_i)$. Let us consider the case $\|\bOmega_1 - \bOmega_2\|_2^2 \leq \delta^2$. Then,
\begin{eqnarray}
\mathop{\max}_{i} |d_{i} - 1| &=&  \|\bmA - \bmI_p\|_{(2,2)} \nonumber \\
&=& \|\bOmega_1^{-1/2}(\bOmega_2 - \bOmega_1)\bOmega_1^{-1/2}\|_{(2,2)} \nonumber \\
&\leq & \|\bOmega_1^{-1}\|_{(2,2)}\|\bOmega_2 - \bOmega_1\|_{(2,2)} \nonumber \\
&\leq & \varepsilon_0^{-1}\delta < 1.
\end{eqnarray}
Now let $h^2(p_{\bOmega_1},p_{\bOmega_2}) \leq 2\delta^2$. This implies $1 - \{\prod_{i=1}^p\frac{1}{2}(d_{i}^{1/2} + d_{i}^{-1/2})\}^{-1/2} \leq \delta^2$. Rearranging the terms, we get,
$\prod_{i=1}^{p}\frac{1}{2}(d_{i}^{1/2} + d_{i}^{-1/2}) \leq (1 - \delta^2)^{-2} = 1 + \eta,\, \mathrm{say}$. Since every term in the product exceeds 1, we have,
\begin{equation}
\label{eqn:dibound}
\mathop{\max}_{i}\frac{1}{2}(d_{i}^{1/2} + d_{i}^{-1/2}) \leq 1 + \eta.
\end{equation}
The above equation, upon squaring and rearrangement of terms, gives, for all $i$,
\begin{equation}
(d_{i} - 1)^2 \leq 2d_{i}^{1/2}\eta.
\end{equation}
Note that equation (\ref{eqn:dibound}) gives that $ d_{i}^{1/2} \leq 2(1+\eta)$. Hence, the above equation implies that $(d_{i} - 1)^2 \leq 4\eta(1+\eta)$. Choose $\eta < (\sqrt{2}-1)/2$ so that we get $|d_{i} - 1| < 1$ for all $i=1,\ldots,p$. 

Now, let us assume that $\frac{1}{2}h^2(p_{\bOmega_1},p_{\bOmega_2}) \leq \delta^2$, for some $\delta > 0$. This implies, from equation (\ref{eqn:hellingerp1p2}),
\begin{equation*}
\prod_{i=1}^p(d_{i}^{1/2}+d_{i}^{-1/2}) \leq 2^p(1-\delta^2)^{-2}.
\end{equation*}
Now, $\prod_{i=1}^p(d_{i}^{1/2}+d_{i}^{-1/2}) = 2^p[1+O\{\sum_{i=1}^p(d_{i}-1)^2\}]$ using Taylor's series expansion. Then, from the above equation after rearrangement of the terms, we get, $1 + O\left\{\sum_{i=1}^p(d_{i}-1)^2\right\} \leq (1-\delta^2)^{-2} \sim 1 + 2\delta^2$, so that,
\begin{equation}
\sum_{i=1}^p(d_{i}-1)^2 \leq c_0\delta^2,\, \mbox{for some} \, c_0 > 0.
\end{equation}
Now, equation (\ref{eqn:Frobbound}) gives that $\|\bOmega_1 - \bOmega_2\|_2^2 \leq \|\bOmega_1\|^2_{(2,2)}\sum_{i=1}^p(1 - d_{i})^2$. Choosing $\delta = h(p_{\bOmega_1},p_{\bOmega_2})$, the first inequality follows.

To show the other way round, assume that $\|\bOmega_1 - \bOmega_2\|_2^2 \leq \delta^2$. Then,
\begin{eqnarray}
\frac{1}{2}h^2(p_{\bOmega_1},p_{\bOmega_2}) &=& 1 - \frac{\prod_{i=1}^p d_{i}^{-1/4}}{\{\prod_{i=1}^p\frac{1}{2}(1 + d_{i}^{-1})\}^{1/2}} \nonumber \\
&=& 1 - \frac{1}{[1 + O\left\{\sum_{i=1}^p(d_{i}-1)^2\right\}]^{1/2}}
 \nonumber \\
&= & O\left\{\sum_{i=1}^p(d_{i}-1)^2\right\} \leq c\delta^2,\, \mbox{for some} \, c > 0.
\end{eqnarray}
Thus, if $\|\bOmega_1 - \bOmega_2\|_2^2 \leq \delta^2$, then  $h^2(p_{\bOmega_1},p_{\bOmega_2}) \leq c\delta^2$ for some $c > 0$.
\end{proof}

We now give a proof of the result on the graphical lasso solutions being identical in case of a regular submodel for a non-regular model.
\begin{proof}[Proof of Lemma \ref{lemma:nonregular}]
The graphical lasso solution for the model $\bGamma$, given by $\bOmega^*_{\bGamma} =  (\!(\omega^*_{\bGamma,ij})\!)$ satisfies, by the Karush-Kuhn-Tucker (KKT) condition; see, for example, \cite{boyd2004convex}, \cite{witten2011new});
\begin{equation}
\bOmega^{*-1}_{\bGamma} - \widehat{\bSigma} - \lambda\bm{G} = 0,
\end{equation}
where $\bm{G}$ is a matrix with elements
\begin{equation}
\bm{G}_{ij} = \begin{cases}  \omega^*_{\bGamma,ij}/|\omega^*_{\bGamma,ij}| &\mbox{if } \omega^*_{\bGamma,ij} \neq 0 \\
g_{ij} \in [-1,1] & \mbox{if } \omega^*_{\bGamma,ij} = 0. \end{cases}
\end{equation}
For a non-regular model $\bGamma$, consider the non-zero elements $\omega^*_{\bGamma,ij}$ of the graphical lasso solution. Corresponding to the submodel $\Gamma'$ of $\Gamma$, we construct a matrix $\bOmega^*_{\bGamma'} =  (\!(\omega_{\bGamma',ij})\!)$ such that,
\begin{equation}
\omega^*_{\bGamma',ij} = \begin{cases} \omega^*_{\bGamma,ij} &\mbox{if } \omega^*_{\bGamma,ij} \neq 0 \\
0 & \mbox{otherwise}. \end{cases}
\end{equation}
Then, $\bOmega^*_{\bGamma'}$ satisfies the KKT condition corresponding to the model $\bGamma'$, and hence $\bOmega^*_{\bGamma'}$ is a graphical lasso solution for $\bGamma'$. But the construction of the above solution gives that $\bOmega^*_{\bGamma} = \bOmega^*_{\bGamma'}$. This completes the proof.
\end{proof}

The following lemma is essential in proving the ignorability of the non-regular models for posterior probability evaluation.
\begin{lemma}
\label{lemma:Nonregularitylemma}
Consider a non-regular model $\bGamma$ with the corresponding regular submodel $\bGamma'$, having identical graphical lasso estimate given by $\bOmega^*$. For $\bDelta_{\bGamma} = (\!(u_{\bGamma,ij})\!)$ as defined in Section \ref{subsec:nonregular}, for fixed values of $u_{\bGamma,ij} \in \{u_{\bGamma,ij} : \gamma_{ij} = \gamma'_{ij} = 1, \|\bDelta_{\bGamma}\|_2 \leq \epsilon_n\}$, we have,
\begin{equation}
\log \mathrm{det}(\bDelta_{\bGamma} + \bOmega^*) - \mathrm{tr}(\widehat{\bSigma}\bDelta_{\bGamma}) \leq \log \mathrm{det}(\bDelta_{\bGamma'} + \bOmega^*) - \mathrm{tr}(\widehat{\bSigma}\bDelta_{\bGamma'}).
\end{equation}
\end{lemma}
\begin{proof}
Consider maximization of the function
\begin{equation}
f(\bDelta_{\bGamma}) = \log \det (\bDelta_{\bGamma} + \bOmega^*) - \tr(\widehat{\bSigma}\bDelta_{\bGamma}).
\end{equation}
with respect to the elements $u_{\bGamma,ij}$ where $(i,j) \in \{(i,j): \gamma_{ij} = 1, \gamma'_{ij} = 0 \}$. Differentiating the above function for a particular value of $u_{ij}$ gives,
\begin{equation}
\frac{\partial f(\bDelta_{\bGamma})}{\partial u_{\bGamma,ij}} = \tr\left[\left\{(\bDelta_{\bGamma} + \bOmega^*)^{-1}\bm{E}_{(i,j)} - \widehat{\bSigma}\bm{E}_{(i,j)}\right\}\right].
\end{equation}
The maximizer $\widehat{u}_{\bGamma,ij}$ satisfies $\tr\left[\left\{(\bDelta_{\bGamma} + \bOmega^*)^{-1}\bm{E}_{(i,j)} - \widehat{\bSigma}\bm{E}_{(i,j)}\right\}\right] = 0$.
Now consider the function $g(\bDelta_{\bGamma})$ as defined earlier. The derivative of $g(\bDelta_{\bGamma})$ with respect to $u_{\bGamma,ij}$ satisfies
\begin{equation}
\left. \frac{\partial g(\bDelta_{\bGamma})}{\partial u_{\bGamma,ij}}\right|_{u_{\bGamma,ij} = 0^+, u_{\bGamma,lm}=0, \forall (l,m)\neq (i,j)} \geq 0,
\end{equation}
and,
\begin{equation}
\left. \frac{\partial g(\bDelta_{\bGamma})}{\partial u_{\bGamma,ij}}\right|_{u_{\bGamma,ij} = 0^-, u_{\bGamma,lm}=0, \forall (l,m)\neq (i,j)} \leq 0.
\end{equation}
The above two conditions give,
\begin{eqnarray*}
\left.\tr\left[\left\{(\bDelta_{\bGamma} + \bOmega^*)^{-1}\bm{E}_{(i,j)} - \widehat{\bSigma}\bm{E}_{(i,j)}\right\}\right]\right|_{u_{\bGamma,ij} = 0^+, u_{\bGamma,lm}=0, \forall (l,m)\neq (i,j)=0} &\leq& \frac{2\lambda}{n}, \\
\left.\tr\left[\left\{(\bDelta_{\bGamma} + \bOmega^*)^{-1}\bm{E}_{(i,j)} - \widehat{\bSigma}\bm{E}_{(i,j)}\right\}\right]\right|_{u_{\bGamma,ij} = 0^-, u_{\bGamma,lm}=0, \forall (l,m)\neq (i,j) = 0} &\geq& - \frac{2\lambda}{n}.
\end{eqnarray*}
If the first derivative of $f(\bDelta_{\bGamma})$ is continuous at $0$, then, we have $\widehat{u}_{\bGamma,ij} = 0$.
This immediately implies the result stated in the lemma.
\end{proof}

We now prove the result on the bound of the remainder term in the Taylor series expansion of the function $h(\bOmega)$. The proof requires other additional results which are also stated and proved below.
\begin{proof}[Proof of Lemma \ref{lemma:Rnbound}]
The Taylor series expansion of $h(\bOmega)$ gives,
\begin{equation}
\label{eqn:TswithRem}
h(\bOmega) = h(\bOmega^*) + \frac{1}{2}\mathrm{vec}(\bDelta)^T \bm{H}_{\bOmega^*}\mathrm{vec}(\bDelta) + R_n,
\end{equation}
where $R_n$ is the remainder term in the expansion. Using the integral form of the remainder, we have,
\begin{equation}
\label{eqn:TswithIntform}
h(\bOmega) = h(\bOmega^*) + \mathrm{vec}(\bDelta)^T \left\{\int_{0}^{1}(1-\nu)\bm{H}_{\bOmega^* + \nu \bDelta}d\nu\right\}\mathrm{vec}(\bDelta).
\end{equation}
Subtracting (\ref{eqn:TswithIntform}) from (\ref{eqn:TswithRem}) gives,
\begin{eqnarray}
R_n &=& \mathrm{vec}(\bDelta)^T \left\{\int_{0}^{1}(1-\nu)\bm{H}_{\bOmega^* + \nu \bDelta}d\nu\right\}\mathrm{vec}(\bDelta) - \frac{1}{2}\mathrm{vec}(\bDelta)^T \bm{H}_{\bOmega^*}\mathrm{vec}(\bDelta) \nonumber \\
&=& \mathrm{vec}(\bDelta)^T \left\{\int_{0}^{1}(1-\nu)\left(\bm{H}_{\bOmega^* + \nu \bDelta} - \bm{H}_{\bOmega^*}\right)d\nu \right\}\mathrm{vec}(\bDelta) \nonumber \\
&\leq & \|\bDelta\|_2^2 \left\|\int_{0}^{1}(1-\nu)\left(\bm{H}_{\bOmega^* + \nu \bDelta} - \bm{H}_{\bOmega^*}\right)d\nu \right\|_{(2,2)} \nonumber \\
&\leq & \|\bDelta\|_2^2 \int_{0}^{1}(1-\nu)\|\bm{H}_{\bOmega^* + \nu \bDelta} - \bm{H}_{\bOmega^*}\|_{(2,2)}d\nu \nonumber \\
& \leq & \frac{1}{2}\|\bDelta\|_2^2 \mathop{\max }_{0\leq \nu \leq 1}\|\bm{H}_{\bOmega^* + \nu \bDelta} - \bm{H}_{\bOmega^*}\|_{(2,2)} \nonumber \\
& \leq & \frac{1}{2}\|\bDelta\|_2^2 (p+s)\mathop{\max }_{0\leq \nu \leq 1}\|\bm{H}_{\bOmega^* + \nu \bDelta} - \bm{H}_{\bOmega^*}\|_{\infty} \label{boundforrem}.
\end{eqnarray}
The above bound involves the maximum of the absolute differences between the elements of the Hessian matrices $\bm{H}$ computed at two different values $\bOmega^* + \nu\bDelta$ and $\bOmega^*$. We first show that, with probability tending to one,
 \begin{equation}
 \label{eqn:omegaelmntdiff}
\|(\bOmega^* + \nu\bDelta)^{-1} - \bOmega^{*-1}\|_{\infty} \leq K\|\bDelta\|_2
\end{equation}
Using the matrix norm relations in (\ref{eqn:matrixnorm}), we get,
\begin{eqnarray}
\|(\bOmega^* + \nu\bDelta)^{-1} - \bOmega^{*-1}\|_{(2,2)} &=& \|\left((\bmI + \nu\bOmega^{*-1}\bDelta)^{-1} - \bmI\right)\bOmega^{*-1}\|_{(2,2)} \nonumber \\
&=& \|\nu\bOmega^{*-1}\bDelta(\bmI + \nu\bOmega^{*-1}\bDelta)^{-1}\bOmega^{*-1}\|_{(2,2)} \nonumber \\
&\leq & \nu \|\bOmega^{*-1}\|^2_{(2,2)}\|\bDelta\|_{(2,2)} \nonumber \\
&\leq & \|\bOmega^{*-1}\|^2_{(2,2)}\|\bDelta\|_2 \leq K\|\bDelta\|_2,
\end{eqnarray}
with probability tending to 1, using (\ref{eqn:glassoeigbound}).
Thus, noting that $\|(\bOmega^* + \nu\bDelta)^{-1} - \bOmega^{*-1}\|_{\infty} \leq \|(\bOmega^* + \nu\bDelta)^{-1} - \bOmega^{*-1}\|_{(2,2)}$, we prove the result in equation (\ref{eqn:omegaelmntdiff}). 

For any symmetric matrix $\bmA$ of order $d$, we note that $\mathrm{tr}\left\{\bmA\bm{E}_{(i,j)} \bmA\bm{E}_{(l,m)}\right\}$ has the form $a_1a_2 + a_3a_4 + a_5a_6 + a_7a_8$ where $i,j,l,m \in \{1,\ldots,d\}$, and $a_j$s are some elements of $\bmA$. This can be derived easily by writing out the elements of the product of the matrices involved and noting that matrices like $\bm{E}_{(i,j)}$ have non-zero entries at only two places corresponding to $(i,j)$. Hence the elements of $\bm{H}_{\bOmega^* + \nu \bDelta} - \bm{H}_{\bOmega^*}$ have the form $(a_1a_2 + a_3a_4 + a_5a_6 + a_7a_8) - (b_1b_2 + b_3b_4 + b_5b_6 + b_7b_8)$, where $a_j$'s and $b_j$'s are some elements of $(\bOmega^* + \nu\bDelta)^{-1}$ and $\bOmega^{*-1}$ respectively. Then, using equation (\ref{eqn:omegaelmntdiff}), we get, with probability tending to one,
\begin{equation}
\label{eqn:Hesselementdiff}
\sum a_1a_2 - \sum b_1b_2 \leq C_1\|\bDelta\|_2\|\bOmega^{*-1}\|_{\infty} + C_2 \|\bDelta\|_2^2.
\end{equation}
Since this holds true for any arbitrary element of $\bm{H}_{\bOmega^* + \nu \bDelta} - \bm{H}_{\bOmega^*}$, using (\ref{eqn:glassoeigbound}) and (\ref{eqn:Hesselementdiff}), we get that with probability tending to one,
\begin{equation}
\label{eqn:maxdiffHess}
\|\bm{H}_{\bOmega^* + \nu \bDelta} - \bm{H}_{\bOmega^*}\|_{\infty} \leq C_1\|\bDelta\|_2 + C_2 \|\bDelta\|_2^2,
\end{equation}
where $C_1$ and $C_2$ are suitable constants.

Using (\ref{boundforrem}) and (\ref{eqn:maxdiffHess}), with probability tending to one, we have,
\begin{equation*}
R_n \leq \frac{1}{2}(p+s)\|\bDelta\|_2^2\left(C_1\|\bDelta\|_2 + C_2 \|\bDelta\|_2^2\right).
\end{equation*}
\end{proof}

We now prove the result on the error in Laplace approximation of the posterior probabilities of graphical model structures.
\begin{proof}[Proof of Theorem \ref{theorem:appxerror}]
Using the Taylor series expansion of $h(\bOmega)$ as in (\ref{eqn:TswithRem}), we can write the posterior probability of the graphical structure indicator $\bGamma$ given the data $\bX^{(n)}$ as in equation (\ref{eqn:postprob}) to be proportional to
\begin{equation}
\int_{\bDelta + \bOmega^* \in \mathcal{M}^+} \exp\left\{-\frac{n}{2}\left(h(\bOmega^*) + \frac{1}{2}\mathrm{vec}(\bDelta)^T \bm{H}_{\bOmega^*}\mathrm{vec}(\bDelta) + R_n\right)\right\} \prod_{(i,j) \in \mathcal{V}_{\bGamma}}du_{ij}.
\end{equation}
We denote $\prod_{(i,j) \in \mathcal{V}_{\bGamma}}du_{ij}$ by $d\bDelta$ for notational simplicity. Using (\ref{eqn:postratio}), we get
\begin{equation}
\frac{\int_{\|\bDelta\|_2 \leq \epsilon_n} \exp\left[-\frac{n}{2}\left\{h(\bOmega^*) + \frac{1}{2}\mathrm{vec}(\bDelta)^T \bm{H}_{\bOmega^*}\mathrm{vec}(\bDelta) + R_n\right\}\right] d\bDelta}{\int_{\bDelta + \bOmega^* \in \mathcal{M}^+} \exp\left[-\frac{n}{2}\left\{h(\bOmega^*) + \frac{1}{2}\mathrm{vec}(\bDelta)^T \bm{H}_{\bOmega^*}\mathrm{vec}(\bDelta) + R_n\right\}\right] d\bDelta} \rightarrow 1.
\end{equation}
Also, for $\|\bDelta\|_2 \leq \epsilon_n$, $R_n \leq (p+s)\|\bDelta\|_2^2\epsilon_n/2$. Thus, the upper and lower bounds of the integral $\int_{\|\bDelta\|_2 \leq \epsilon_n}\exp\left\{-\frac{n}{2}h(\bOmega)\right\}d\bDelta$ are given by
\begin{eqnarray}
\label{eqn:intbounds}
\lefteqn{
e^{-nh(\bOmega^*)/2}\int_{\|\bDelta\|_2 \leq \epsilon_n } \exp\left\{-\frac{n}{2}\left(\frac{1}{2}\mathrm{vec}(\bDelta)^T \bm{H}_{\bOmega^*}\mathrm{vec}(\bDelta) \mp \frac{1}{2}(p+s)\epsilon_n\|\bDelta\|_2^2\right)\right\} d\bDelta.} \nonumber \\
&=& e^{-nh(\bOmega^*)/2}\int_{\|\bDelta\|_2 \leq \epsilon_n } \exp\left[-\frac{n}{4}\mathrm{vec}(\bDelta)^T \left\{\bm{H}_{\bOmega^*} \mp (p+s)\epsilon_n\bmI\right\}\mathrm{vec}(\bDelta)\right] d\bDelta. \nonumber \\
\end{eqnarray}
Note that,
\begin{equation}
\int_{\|\bDelta\|_2 > \epsilon_n } \exp\left[-\frac{n}{4}\mathrm{vec}(\bDelta)^T \left\{\bm{H}_{\bOmega^*} \mp (p+s)\epsilon_n\bmI\right\}\mathrm{vec}(\bDelta)\right] d\bDelta \rightarrow 0,
\end{equation}
if $(p+s)\epsilon_n \rightarrow 0$ and the minimum eigenvalue of $\bm{H}_{\bOmega^*}$ is bounded away from zero, which we prove in Lemma \ref{lemma:Hesseigbound} below.
Hence, the bounds can be simplified to 
\begin{equation}
e^{-nh(\bOmega^*)/2}\int_{\bDelta + \bOmega^* \in \mathcal{M}^+} \exp\left[-\frac{n}{4}\mathrm{vec}(\bDelta)^T \left\{\bm{H}_{\bOmega^*} \mp (p+s)\epsilon_n\bmI\right\}\mathrm{vec}(\bDelta)\right] d\bDelta.
\end{equation}
Using the above bounds, the ratio of the actual integral to the approximate integral has upper and lower bounds given by
\begin{eqnarray}
&&\frac{\int_{\bDelta + \bOmega^* \in \mathcal{M}^+} \exp\left[-\frac{n}{4}\mathrm{vec}(\bDelta)^T \left\{\bm{H}_{\bOmega^*} \mp (p+s)\epsilon_n\bmI\right\}\mathrm{vec}(\bDelta)\right] d\bDelta}{\int_{\bDelta + \bOmega^* \in \mathcal{M}^+} \exp\left\{-\frac{n}{4}\mathrm{vec}(\bDelta)^T \bm{H}_{\bOmega^*}\mathrm{vec}(\bDelta)\right\} d\bDelta} \nonumber \\
& = & \left[\frac{\det\left\{\bm{H}_{\bOmega^*} \pm (p+s)\epsilon_n\bmI_p\right\}}{\det(\bm{H}_{\bOmega^*})}\right]^{-1/2}.
\end{eqnarray}
The above expression is bounded between $\left[1 \mp \{\eig_1(\bm{H}_{\bOmega^*})\}^{-1}(p+s)\epsilon_n \right]^{-(p+s)/2}$. Using Lemma \ref{lemma:Hesseigbound} below, $\eig_1( \bm{H}_{\bOmega^*}) \gg 0$, and hence the above bound on the ratio goes to 1 if $(p+s)^2\epsilon_n \rightarrow 0$, so that the error in Laplace approximation is asymptotically small.

\end{proof}

We now prove the result that the eigenvalues of the Hessian $\bm{H}_{\bOmega^*}$ are bounded away from zero.
\begin{lemma}
\label{lemma:Hesseigbound}
Given a graphical model with model indicator $\bGamma$, the minimum eigenvalue of the Hessian $\bm{H}_{\bOmega^*}$ corresponding to the function $h(\bOmega)$, evaluated at $\bOmega^*$, is bounded away from zero.
\end{lemma}
\begin{proof}
The Hessian of the function $h(\bOmega)$ corresponding to the full model with $p + \binom{p}{2}$ free elements has the form $\bm{H}_{\bOmega,\mathrm{full}} =  \bOmega^{-1} \otimes \bOmega^{-1}$. The Hessian $\bm{H}_{\bOmega^*}$ evaluated at the graphical lasso solution $\bOmega^*$ corresponding to the graphical model with model indicator $\bGamma$ is a principal minor of $\bm{H}_{\bOmega^*,\mathrm{full}}$. Hence it suffices to prove that the minimum eigenvalue of $ (\bOmega^*)^{-1} \otimes (\bOmega^*)^{-1}$ is bounded away from zero. Note that,
$\eig_1\{(\bOmega^*)^{-1} \otimes (\bOmega^*)^{-1}\} = [\eig_1\{(\bOmega^*)^{-1}\}]^2$. Thus, using (\ref{eqn:glassoeigbound}), $[\eig_1\{(\bOmega^*)^{-1}\}]^2 = 1/\|\bOmega^*\|^2_2 > 0$.
\end{proof}

\bibliographystyle{apalike}
\bibliography{bayesiangraphlassobib}

\end{document}